\documentclass{amsart}
\usepackage{amsfonts}
\usepackage{amsmath,amscd,amsthm}
\usepackage{amssymb}
\usepackage{mathrsfs}

\usepackage[all,cmtip]{xy}
\usepackage{multirow}
\usepackage{enumerate}
\usepackage{stmaryrd}
\usepackage[pdf]{pstricks}

\psset{unit=1pt}
\psset{arrowsize=4pt 1}
\psset{linewidth=.5pt}

\newtheorem{theorem}{Theorem}[section]
\newtheorem{lemma}[theorem]{Lemma}
\newtheorem{proposition}[theorem]{Proposition}
\newtheorem{corollary}[theorem]{Corollary}

\theoremstyle{definition}
\newtheorem{remark}[theorem]{Remark}
\newtheorem{example}[theorem]{Example}

\def\A{\mathbb{A}}
\def\P{\mathbb{P}}

\def\Z{\mathbb{Z}}

\def\L{\mathbb{L}}

\DeclareMathOperator{\im}{im}

\DeclareMathOperator{\divisor}{div}
\DeclareMathOperator{\Pic}{Pic}
\DeclareMathOperator{\Div}{Div}
\DeclareMathOperator{\CDiv}{CDiv}

\DeclareMathOperator{\Sym}{Sym}

\DeclareMathOperator{\pt}{pt}
\DeclareMathOperator{\tr}{tr}
\DeclareMathOperator{\Rees}{Rees}
\DeclareMathOperator{\gr}{gr}
\DeclareMathOperator{\pr}{pr}
\DeclareMathOperator{\charact}{char}
\DeclareMathOperator{\cone}{cone}

\begin{document}

\title{Formal Group Rings of Toric Varieties}
\author{Wanshun Wong}
\address{Department of Mathematics and Statistics\\ University of Ottawa \\ Ontario \\ Canada}
\thanks{The author is supported by the NSERC Discovery grant 385795-2010 and NSERC DAS grant 396100-2010 of Kirill Zaynullin}
\email{wanshunwong@gmail.com}
\date{}
\begin{abstract}
In this paper we use formal group rings to construct an algebraic model of the $T$-equivariant oriented cohomology of smooth toric varieties. Then we compare our algebraic model with known results of equivariant cohomology of toric varieties to justify our construction. Finally we construct the algebraic counterpart of the pull-back and push-forward homomorphisms of blow-ups.
\end{abstract}

\maketitle

\section{Introduction}

Let $\mathsf{h}$ be an algebraic oriented cohomology theory in the sense of Levine-Morel \cite{LM}, where examples include the Chow group of algebraic cycles modulo rational equivalence, algebraic $K$-theory, connective $K$-theory, elliptic cohomology, and a universal such theory called algebraic cobordism. It is known that to any $\mathsf{h}$ one can associate a one-dimensional commutative formal group law $F$ over the coefficient ring $R = \mathsf{h}(\pt)$, given by
\[
c_1(\mathscr{L}_1 \otimes \mathscr{L}_2) = F(c_1(\mathscr{L}_1), c_1(\mathscr{L}_2))
\]
for any line bundles $\mathscr{L}_1, \mathscr{L}_2$ on a smooth variety $X$, where $c_1$ is the first Chern class.

Let $T$ be a split algebraic torus which acts on a smooth variety $X$. An object of interest is the $T$-equivariant cohomology ring $\mathsf{h}_T(X)$ of $X$, and we would like to build an algebraic model for it. More precisely, instead of using geometric methods to compute $\mathsf{h}_T(X)$, we want to use algebraic methods to construct another ring $A$ (our algebraic model), such that $A$ should be easy to compute and it gives information about $\mathsf{h}_T(X)$. Of course the best case is that $A$ and $\mathsf{h}_T(X)$ are isomorphic. 

Let $T^*$ be the character lattice of our torus $T$. A new combinatorial object, called a formal group ring and denoted by $R \llbracket T^* \rrbracket_F$, can then be defined by using $R, F$ and $T^*$. It serves as an algebraic model of the completed $T$-equivariant cohomology ring $\mathsf{h}_T(\pt)^\wedge$ of a point, or of the cohomology ring of the classifying space $\mathsf{h}(BT)$ of $T$. Various computations can then be performed on formal group rings to provide algebraic models of the (usual) cohomology ring $\mathsf{h}(G/B)$ and the $T$-equivariant cohomology ring $\mathsf{h}_T(G/B)$ of a homogeneous space $G/B$, where $G$ is a split semisimple linear algebraic group with a maximal split torus $T$, and $B$ is a Borel subgroup containing $T$. We refer to \cite{CPZ}, \cite{CZZ1} and \cite{CZZ2} for details.

The main goal of this paper is to apply the techniques of formal group rings to a smooth toric variety $X$, so that we obtain algebraic models of the usual cohomology and the $T$-equivariant cohomology of $X$. Since any toric variety can be constructed by gluing affine toric varieties together, our idea is to find a model of the $T$-equivariant cohomology for each affine toric variety and then ``glue" them together. This paper is organized as follows. First we establish notation and recall basic facts on toric varieties and formal group rings in Section 2. Then in Section 3 we prove our main result of gluing formal group rings together, and the output of gluing will be our models of the usual cohomology and the $T$-equivariant cohomology of $X$. Next we compare our models with known results of equivariant cohomology in Section 4. Finally in Section 5 we define the algebraic counterpart of the pull-back and push-forward homomorphisms of blow-ups.

\section{Notation and Preliminaries}
\subsection{Toric Varieties}
Our references for the theory of toric varieties are \cite{CLS} and \cite{F}, and also \cite{ELFST} for the definition of a toric variety over an arbitrary base.

Let $T$ be a split torus over our base field $k$. The character and cocharacter lattices of $T$ are denoted by $T^*$ and $T_*$ respectively, and there is a perfect pairing $\langle \mbox{ }, \mbox{ } \rangle: T^* \times T_* \rightarrow \Z$. A toric variety $X$ is a normal variety on which the split torus $T$ acts faithfully with an open dense orbit. Recall that $X$ is determined by its fan $\Sigma$ in the lattice $T_*$. In this paper we will always assume $X$ is smooth unless otherwise stated, so that every cone $\sigma$ in $\Sigma$ is generated by a subset of a basis of $T_*$. Let $\Sigma_{max}$ be the set of maximal cones of $\Sigma$. The set of all rays (i.e. one-dimensional cones) in $\Sigma$ is denoted by $\Sigma(1)$, and similarly the set of all rays in $\sigma$ is denoted by $\sigma(1)$ for every cone $\sigma \in \Sigma$.

For every cone $\sigma \in \Sigma$, let $U_\sigma$ be the associated open affine subscheme of $X$, and $O_\sigma$ be the $T$-orbit corresponding to $\sigma$ under the Orbit-Cone Correspondence. The stabilizer of any geometric point of $O_\sigma$ is a subtorus $T_\sigma \subseteq T$, so that $O_\sigma \cong T/T_\sigma$. The character and cocharacter lattices of $T_\sigma$ are given by $T_\sigma^* = T^*/\sigma^\perp$ and $T_{\sigma *} = \langle \sigma \rangle = \sigma + (-\sigma) \subseteq N$ respectively, and $\dim T_\sigma = \dim \sigma$. Finally, for every ray $\rho \in \Sigma(1)$, the unique generator of the monoid $\rho$ is denoted by $v_\rho \in T_*$.

\subsection{Formal Group Rings}
Our main reference for formal group rings is \cite{CPZ}.

Let $R$ be a commutative ring, and $F$ be a formal group law over $R$ which we always assume to be one-dimensional and commutative, i.e. $F(x,y) \in R\llbracket x, y \rrbracket$ is a formal power series such that
\[
F(x,0) = 0, \mbox{ } F(x,y) = F(y,x), \mbox{ and } F(x,F(y,z)) = F(F(x,y),z),
\]
see \cite[p.4]{LM}. For any nonnegative integer $n$ we use the notation
\[
x +_F y = F(x,y), \mbox{ } n \cdot_F x = \underbrace{x +_F \cdots +_F x}_{n \mbox{ copies}}, \mbox{ and } (-n) \cdot_F x = -_F (n \cdot_F x)
\]
where $-_F x$ denotes the formal inverse of $x$, i.e. $x +_F (-_F x) = 0$.

Let $M$ be an abelian group, and let $R[x_M]$ denote the polynomial ring over $R$ with variables indexed by $M$. Let $\epsilon: R[x_M] \rightarrow R$ be the augmentation map which sends $x_\lambda$ to $0$ for every $\lambda \in M$. We denote $R \llbracket x_M \rrbracket$ the $\ker(\epsilon)$-adic completion of the polynomial ring $R[x_M]$. Let $J_F \subseteq R \llbracket x_M \rrbracket$ be the closure of the ideal generated by $x_0$ and $x_{\lambda + \mu} - (x_\lambda +_F x_\mu)$ over all $\lambda, \mu \in M$. The formal group ring (also called formal group algebra) is then defined to be the quotient
\[
R \llbracket M \rrbracket_F = R \llbracket x_M \rrbracket / J_F.
\]
By abuse of notation the class of $x_\lambda$ in $R \llbracket M \rrbracket_F$ is also denoted by $x_\lambda$. By definition $R \llbracket M \rrbracket_F$ is a complete Hausdorff $R$-algebra with respect to the $\ker(\epsilon')$-adic topology, where $\epsilon': R \llbracket M \rrbracket_F \rightarrow R$ is the induced augmentation map.

An important subring of $R \llbracket M \rrbracket_F$ is the image of $R[x_M]$ under the composition $R[x_M] \rightarrow R \llbracket x_M \rrbracket \rightarrow R \llbracket M \rrbracket_F$, denoted by $R[M]_F$. Then $R \llbracket M \rrbracket_F$ is the completion of $R[M]_F$ at the ideal $\ker(\epsilon') \cap R[M]_F$.

\begin{example}
(see \cite[Example 2.19]{CPZ}) The additive formal group law over $R$ is given by $F(x,y) = x + y$. In this case we have $R$-algebra isomorphisms
\[
R[M]_F \cong \Sym_R(M) \mbox{ and } R \llbracket M \rrbracket_F \cong \Sym_R(M)^\wedge,
\]
where $\Sym_R(M)$ is the ring of symmetric powers of $M$ over $R$, and the completion is at the kernel of the augmentation map $x_\lambda \mapsto 0$. The isomorphisms are given by $x_\lambda \mapsto \lambda \in \Sym_R(M)$.
\end{example}

\begin{example}
(see \cite[Example 2.20]{CPZ}) The multiplicative periodic formal group law over $R$ is given by $F(x,y) = x + y - \beta xy$, where $\beta$ is a unit in $R$. Let $R[M]$ be the (usual) group ring $R[M] = \{\sum r_i e^{\lambda_i} \mbox{ } | \mbox { } r_i \in R, \lambda_i \in M\}$ written in the exponential notation, and $\tr: R[M] \rightarrow R$ be the trace map which sends $e^\lambda$ to $0$ for every $\lambda \in M$. Then there are the following $R$-algebra isomorphisms
\[
R[M]_F \cong R[M] \mbox{ and } R \llbracket M \rrbracket_F \cong R[M]^\wedge
\]
where the completion is at $\ker(\tr)$, and the isomorphisms are given by sending $x_\lambda$ to $\beta^{-1}(1 - e^\lambda)$.
\end{example}

Finally we remark that given $\phi: M \rightarrow M'$ a homomorphism of abelian groups, it induces ring homomorphisms $R \llbracket M \rrbracket_F \rightarrow R \llbracket M' \rrbracket_F$ and $R[M]_F \rightarrow R[M']_F$ by sending $x_\lambda$ to $x_{\phi(\lambda)}$.

\section{Formal Group Rings of Toric Varieties}
Let $X$ be the smooth toric variety of the fan $\Sigma \subseteq T_*$. For every cone $\sigma \in \Sigma$, since $\sigma$ is smooth we have $U_\sigma = \A^{\dim \sigma} \times T/T_\sigma$. Therefore if $F$ is the formal group law associated to some oriented cohomology theory $\mathsf{h}$ over the coefficient ring $R = \mathsf{h}(\pt)$, by homotopy invariance the formal group ring $R \llbracket T_\sigma^* \rrbracket_F$ can be viewed as an algebraic substitute of the completed equivariant cohomology ring $\mathsf{h}_T(U_\sigma)^\wedge$, see \cite[Remark 2.22]{CPZ}. It is known that a topology can be given to the fan $\Sigma$ by defining the open sets to be subfans of $\Sigma$ (see for example \cite[Section 7.2]{GK}). Our goal is to ``glue" $R \llbracket T_\tau^* \rrbracket_F$ over all maximal cones $\tau \in \Sigma_{max}$ together as a sheaf on $\Sigma$, and the ring of global sections of this sheaf will then be an algebraic model of $\mathsf{h}_T(X)^\wedge$.

For any toric variety $X$, there is a natural isomorphism
\begin{equation}
\label{cdivisom}
\CDiv_T(X) \cong \ker \left(\bigoplus_{\tau \in \Sigma_{max}} T_{\tau}^* \longrightarrow \bigoplus_{\tau \neq \tau'} T_{\tau \cap \tau'}^*\right)
\end{equation}
where $\CDiv_T(X)$ is the group of $T$-invariant Cartier divisors, and the map on the right hand side is given by the difference of the two natural projections $T_{\tau}^*, T_{\tau'}^* \rightarrow T_{\tau \cap \tau'}^*$ on each summand. The idea is that for every $T$-invariant Cartier divisor, its restriction to $U_{\tau}$ is equal to the divisor of a character, and the obvious compatability condition holds (see \cite[Chapter 4.2]{CLS}). 

For every ray $\rho \in \Sigma(1)$, the closure of the corresponding orbit $\overline{O_\rho}$ is a $T$-invariant prime divisor on $X$, and we will denote it by $D_\rho$. The group of $T$-invariant Weil divisors $\Div_T(X)$ is a lattice generated by $D_\rho$, i.e.
\[
\Div_T(X) = \bigoplus_{\rho \in \Sigma(1)} \Z \cdot D_\rho.
\]

Since we assume $X$ is smooth, all Weil divisors are Cartier. Hence
\[
\CDiv_T(X) = \Div_T(X) = \bigoplus_{\rho \in \Sigma(1)} \Z \cdot D_\rho.
\]
For every character $\alpha \in T^*$, it determines a $T$-invariant principal Cartier divisor $\divisor(\alpha) = \sum_{\rho \in \Sigma(1)} \langle \alpha, v_\rho \rangle D_\rho$. This defines a group homomorphism $T^* \rightarrow \CDiv_T(X)$. Passing to the formal group rings, we have
\[
\begin{tabular}{ccl}
$R\llbracket T^* \rrbracket_F$ & $\longrightarrow$ & $R\llbracket \CDiv_T(X) \rrbracket_F$ \\
$x_\alpha$ & $\longmapsto$ & $x_{\sum \langle \alpha, v_\rho \rangle D_\rho} = \sum [\langle \alpha, v_\rho\rangle]_F x_{D_\rho}$.
\end{tabular}
\]
Therefore $R\llbracket \CDiv_T(X) \rrbracket_F$ is a $R\llbracket T^* \rrbracket_F$-algebra. Clearly, $R\llbracket T_\sigma^* \rrbracket_F$ is also a $R\llbracket T^* \rrbracket_F$-algebra under the natural map for every cone $\sigma \in \Sigma$. 

Consider the group homomorphism $\CDiv_T(X) \rightarrow T_\sigma^*$ defined by 
\[
\xymatrix{\sum_{\rho \in \Sigma(1)} n_\rho D_\rho \ar@{|->}[r] & \sum_{\rho \in \sigma(1)} n_\rho \alpha_{\sigma,\rho}}
\]
where $\{\alpha_{\sigma,\rho} \mbox{ }|\mbox{ } \rho \in \sigma(1)\}$ is the basis of $T_\sigma^*$ dual to $\{v_\rho \mbox{ }|\mbox{ } \rho \in \sigma(1)\}$. We remark that the above homomorphism is different from the one in literature by a minus sign. Again passing to formal group rings we obtain a map $R\llbracket \CDiv_T(X) \rrbracket_F \rightarrow R\llbracket T_\sigma^* \rrbracket_F$.

\begin{lemma}
\label{lemma1}
The map $R\llbracket \CDiv_T(X) \rrbracket_F \rightarrow R\llbracket T_\sigma^* \rrbracket_F$ constructed above is a $R\llbracket T^* \rrbracket_F$-algebra homomorphism.
\end{lemma}

\begin{proof}
By functoriality of formal group rings it suffices to show that 
\[
\xymatrix{\CDiv_T(X) \ar[rr] & & T_\sigma^* \\
& T^* \ar[lu] \ar[ru] & \\}
\]
is commutative. For every $\alpha \in T^*$, we have
\[
\xymatrix{\sum_{\rho \in \Sigma(1)} \langle \alpha, v_\rho \rangle D_\rho \ar@{|->}[rr] & & \sum_{\rho \in \sigma(1)} \langle \alpha, v_\rho \rangle \alpha_{\sigma,\rho} \\
& \alpha \ar@{|->}[lu] \ar@{|->}[ru] & \\}
\]
which is clearly commutative.
\end{proof}

Now we are ready to do the ``gluing". Motivated by the isomorphism \eqref{cdivisom}, we consider the following sequence
\begin{equation}
\label{sheaf}
\xymatrix{R\llbracket \CDiv_T(X) \rrbracket_F  \ar[r]^-\psi & \displaystyle \prod_{\tau \in \Sigma_{max}} R\llbracket T_{\tau}^* \rrbracket_F \ar[r]^-\pi & \displaystyle \prod_{\tau \neq \tau'} R\llbracket T_{\tau \cap \tau'}^* \rrbracket_F}
\end{equation}
where $\psi$ is the product of the maps in Lemma \ref{lemma1}, $\pi$ is given by the product of the differences of $\pr_{\tau, \tau \cap \tau'} : R\llbracket T_{\tau}^* \rrbracket_F \rightarrow R\llbracket T_{\tau \cap \tau'}^* \rrbracket_F$, $\pr_{\tau', \tau \cap \tau'} : R\llbracket T_{\tau'}^* \rrbracket_F \rightarrow R\llbracket T_{\tau \cap \tau'}^* \rrbracket_F$. (Here it involves a choice between $\pr_{\tau, \tau \cap \tau'} - \pr_{\tau', \tau \cap \tau'}$ and $\pr_{\tau', \tau \cap \tau'} - \pr_{\tau, \tau \cap \tau'}$. However we are only interested in the kernel so it does not matter.) Note that $\pi$ is only a module homomorphism but not an algebra homomorphism.

\begin{proposition}
\label{smoothprop}
The sequence \eqref{sheaf} is an exact sequence of $R\llbracket T^* \rrbracket_F$-modules.
\end{proposition}

\begin{proof}
i) First, it follow from the functoriality of formal group rings that $\im(\psi) \subseteq \ker(\pi)$.

ii) To show $\ker(\pi) \subseteq \im(\psi)$, let $\tau_1, \ldots, \tau_d$ be the maximal cones of $\Sigma$, and let $(f_i)$ be any element in $\ker(\pi) \subseteq \prod_i R\llbracket T_{\tau_i}^* \rrbracket_F$. Then we define $f_{ij}$ to be the image of $f_i$ in $R\llbracket T_{\tau_i \cap \tau_j}^* \rrbracket_F$ (which is the same as the image of $f_j$), $f_{ijk}$ to be the image of $f_i$ in $R\llbracket T_{\tau_i \cap \tau_j \cap \tau_k}^* \rrbracket_F$, and so on. 

By \cite[Corollary 2.13]{CPZ}, for every cone $\sigma$ we identify $R\llbracket T_\sigma^* \rrbracket_F$ with the ring of power series $R\llbracket x_{\alpha_{\sigma,\rho}} \rrbracket$ with variables $x_{\alpha_{\sigma,\rho}}, \rho \in \sigma(1)$, where we recall that $\{\alpha_{\sigma,\rho} \mbox{ }|\mbox{ } \rho \in \sigma(1)\}$ is the basis of $T_\sigma^*$ dual to $\{v_\rho \mbox{ }|\mbox{ } \rho \in \sigma(1)\}$. Since $X$ is smooth, under this identification for every face $\mu$ of $\sigma$, the natural maps $R\llbracket T_{\sigma}^* \rrbracket_F \rightarrow R\llbracket T_{\mu}^* \rrbracket_F$ coincides with the canonical projection of the rings of power series $R\llbracket x_{\alpha_{\sigma,\rho}} \rrbracket \rightarrow R\llbracket x_{\alpha_{\mu,\rho}} \rrbracket$, $x_{\alpha_{\sigma,\rho}} \mapsto x_{\alpha_{\mu,\rho}}$ if $\rho \in \mu(1)$, $x_{\alpha_{\sigma,\rho}} \mapsto 0$ if $\rho \notin \mu(1)$. Similarly we identify $R\llbracket \CDiv_T(X) \rrbracket_F$ with $R\llbracket x_{D_\rho} \rrbracket$ with variables $x_{D_\rho}, \rho \in \Sigma(1)$. Then $\psi$ coincides with product of the projections $R\llbracket x_{D_\rho} \rrbracket \rightarrow R\llbracket x_{\alpha_{\tau_i,\rho}} \rrbracket$.

For every $i$, let $g_i$ be the unique preimage of $f_i$ under the projection $R\llbracket x_{D_\rho} \rrbracket \rightarrow R\llbracket x_{\alpha_{\tau_i,\rho}} \rrbracket$, such that $g_i$ does not involve any $x_{D_\rho}$ for $\rho \notin \tau_i(1)$. Informally speaking, $g_i$ is obtained from $f_i$ by replacing all $x_{\alpha_{\tau_i,\rho}}$ with corresponding $x_{D_\rho}$. We define $g_{ij}, \ldots, g_{12\cdots d}$ in the same way.

Finally we define 
\[
g = \sum_i g_i - \sum_{i < j} g_{ij} + \sum_{i < j < k} g_{ijk} - \cdots + (-1)^{d+1} g_{12\cdots d} \in R\llbracket x_{D_\rho} \rrbracket.
\]
Then for example under the projection $R\llbracket x_{D_\rho} \rrbracket \rightarrow R\llbracket x_{\alpha_{\tau_1,\rho}} \rrbracket$,
\begin{multline*}
g = g_1 + \left(\sum_{1<i} g_i - \sum_{1<j} g_{1j}\right) + \left(-\sum_{1<i<j} g_{ij} + \sum_{1 < j < k} g_{1jk}\right) + \cdots + (-1)^{d+1} g_{12\cdots r} \\
\longmapsto f_1 + 0 + 0 + \cdots + 0 = f_1.
\end{multline*}
So $(f_i) = \psi(g) \in \im(\psi)$.
\end{proof}

Under the above identifications of the formal group rings with the rings of power series, we can have an explicit description of $\ker(\psi)$.

\begin{proposition}
\label{srideal}
$\ker(\psi)$ is equal to $I_\Sigma$, the ideal generated by the square-free monomial $\prod_{\rho \in S} x_{D_\rho}$ over all subsets $S \subseteq \Sigma(1)$ such that $S \nsubseteq \sigma(1)$ for any cone $\sigma$.
\end{proposition}

\begin{proof}
i) $\ker(\psi) \supseteq I_\Sigma$: Let $x_{D_{\rho_1}} \cdots x_{D_{\rho_t}}$ be a generator of $I_\Sigma$. For every maximal cone $\tau_i$, by construction $\rho_j \notin \tau_i(1)$ for some $1 \leq j \leq t$. Therefore $x_{D_{\rho_j}} \mapsto 0 \in R\llbracket x_{\alpha_{\tau_i,\rho}} \rrbracket$, and $x_{D_{\rho_1}} \cdots x_{D_{\rho_t}} \mapsto 0$ as well.

ii) $\ker(\psi) \subseteq I_\Sigma$: For every $f \in \ker(\psi)$, write $f$ as a linear combination of monomials in $x_{D_\rho}$. Then for every maximal cone $\tau_i$, $f \mapsto 0 \in R\llbracket x_{\alpha_{\tau_i,\rho}} \rrbracket$ implies each of these monomials also maps to $0$, which means each of them contain some $x_{D_\rho}$ for $\rho \notin \tau_i(1)$. Hence these monomials and $f$ are in $I_\Sigma$.
\end{proof}

\begin{remark}
$I_\Sigma$ is power series version of the standard Stanley-Reisner ideal. The underlying geometric meaning follows from the Orbit-Cone Correspondence: It is known that $\rho$ is a face of $\sigma$ if and only if $O_\sigma \subseteq D_\rho$. Therefore given $S = \{\rho_1, \ldots, \rho_t \} \subseteq \Sigma(1)$, we have $S \nsubseteq \sigma(1)$ for any cone $\sigma$ if and only if $D_{\rho_1} \cap \cdots \cap D_{\rho_t} = \emptyset$ in $X$.
\end{remark}

\begin{corollary}
\label{coro}
$\ker(\pi) = \im(\psi) = R\llbracket x_{D_\rho} \rrbracket/I_\Sigma$.
\end{corollary}

By ``gluing" we have constructed an algebraic model for the completed equivariant cohomology ring $\mathsf{h}_T(X)^\wedge$. Recall that by an algebraic model we mean a ring that can be computated by purely algebraic methods and is closely related to $\mathsf{h}_T(X)^\wedge$.

\begin{theorem}
\label{gluing}
$R\llbracket \CDiv_T(X) \rrbracket_F/I_\Sigma$ is our algebraic model for the completed equivariant cohomology ring $\mathsf{h}_T(X)^\wedge$.
\end{theorem}

\begin{proof}
By Proposition \ref{smoothprop} and Corollary \ref{coro} we have the following exact sequence of $R\llbracket T^* \rrbracket_F$-modules
\[
\xymatrix{0 \ar[r] & R\llbracket \CDiv_T(X) \rrbracket_F/I_\Sigma \ar[r] & \displaystyle \prod_{\tau \in \Sigma_{max}} R\llbracket T_{\tau}^* \rrbracket_F \ar[r] & \displaystyle \prod_{\tau \neq \tau'} R\llbracket T_{\tau \cap \tau'}^* \rrbracket_F}.
\]
Notice that this is precisely the exact sequence of the sheaf axiom, where the subfans induced by $\tau$, $\tau$ varies over $\Sigma_{max}$, form an open covering of $\Sigma$. Hence the $R\llbracket T^* \rrbracket_F$-algebra $R\llbracket \CDiv_T(X) \rrbracket_F/I_\Sigma$ is the ring of global sections and is our algebraic model of $\mathsf{h}_T(X)^\wedge$. We remark that similar exact sequences for equivariant singular cohomology and equivariant $K$-theory can be found in \cite[Chapter 12]{CLS} and \cite{AHW} respectively. 
\end{proof}

\begin{remark}
Let $k$ be a field of characteristic 0 and $\mathsf{h} = \Omega^*$ be the algebraic cobordism. The formal group law associated to $\Omega^*$ is the univeral formal group law, and the coefficient ring is the Lazard ring $\L$. As we see in Example \ref{ufgl}, our algebraic model $\L \llbracket \CDiv_T(X) \rrbracket_F/I_\Sigma$ is isomorphic to $\Omega^*_T(X)$. Since $\Omega^*$ is the universal oriented cohomology theory, it follows from the functoriality of formal group rings that we will have isomorphisms between $R\llbracket \CDiv_T(X) \rrbracket_F/I_\Sigma$ and $\mathsf{h}_T(X)^\wedge$ for all other oriented cohomology theories as well.

One of the main advantages of our construction is that it still works when $k$ is of characteristic $p > 0$. We are still able to construct an algebraic model for $\mathsf{h}_T(X)^\wedge$, while on the other hand there is no universal theory for us to specialize from.
\end{remark}

Next we would like to study the usual cohomology of $X$. Recall that $R$ is a $R\llbracket T^* \rrbracket_F$-algebra via the augmentation map. Then we have an isomorphism
\[
(R\llbracket \CDiv_T(X) \rrbracket_F/I_\Sigma) \otimes_{R\llbracket T^* \rrbracket_F} R \mbox{ } \cong \mbox{ } R\llbracket \CDiv_T(X) \rrbracket_F/J_\Sigma
\]
where $J_\Sigma$ is the ideal generated $I_\Sigma$ and $\sum_{\rho \in \Sigma(1)} [\langle \alpha, v_\rho\rangle]_F x_{D_\rho}$ over all $\alpha \in T^*$. This construction corresponds to the idea that the usual cohomology ring is a quotient of the equivariant cohomology ring, where the corresponding results for Chow group, algebraic $K$-theory and algebraic cobordism are proved in \cite[Corollary 2.3]{B2}, \cite[Proposition 28]{Me} and \cite[Theorem 8.1]{KU} respectively.

It is known that there is the following exact sequence
\[
\xymatrix{T^* \ar[r] & \CDiv_T(X) \ar[r] & \Pic(X) \ar[r] & 0}
\]
where the first homomorphism is defined before Lemma \ref{lemma1}, and the second homomorphism sends a $T$-invariant Cartier divisor to its class in the Picard group. 

\begin{lemma}
\label{lemmapicgp}
$R\llbracket \CDiv_T(X) \rrbracket_F \otimes_{R\llbracket T^* \rrbracket_F} R$ is isomorphic to $R\llbracket \Pic(X) \rrbracket_F$.
\end{lemma}

\begin{proof}
First, if we identify the lattices $T^*$ and $\CDiv_T(X)$ with $\Z^m$ and $\Z^{m'}$ respectively, the homomorphism $T^* \rightarrow \CDiv_T(X)$ is given by a $m' \times m$ matrix with coefficients in $\Z$. Since every matrix with coefficients in $\Z$ has a Smith normal form, it means that we can choose a new $\Z$-basis $\{u_1, \ldots, u_m\}$ of $T^*$ and a new $\Z$-basis $\{u'_1, \ldots, u'_{m'}\}$ of $\CDiv_T(X)$ such that the homomorphism is given by
\[
u_i \longmapsto \begin{cases} a_i u'_i & \mbox{ if } 1 \leq i \leq s \\
0 & \mbox{ if } s + 1 \leq i \leq m \end{cases}
\]
for some integer $s$, and $a_1 | \cdots | a_s$ are positive integers (notice that the value 1 is allowed). Then $\Pic(X)$ is isomorphic to $\Z/a_1 \Z \oplus \cdots \oplus \Z/a_s \Z \oplus \Z^{ m' - s}$. 

By using Theorem 2.11, Corollary 2.13 and Example 2.15 of \cite{CPZ},
\begin{align*}
R\llbracket T^* \rrbracket_F & \cong R\llbracket x_1, \ldots, x_m \rrbracket \\
R\llbracket \CDiv_T(X) \rrbracket_F & \cong R\llbracket x'_1, \ldots, x'_{m'} \rrbracket \\
R\llbracket \Pic(X) \rrbracket_F & \cong R\llbracket x'_1, \ldots, x'_{m'} \rrbracket/\langle a_i \cdot_F x'_i | i = 1, \ldots s \rangle.
\end{align*}
The formal group ring homomorphism induced by $T^* \rightarrow \CDiv_T(X)$ is given by
\[
x_i \longmapsto \begin{cases} a_i \cdot_F x'_i & \mbox{ if } 1 \leq i \leq s \\
0 & \mbox{ if } s + 1 \leq i \leq m, \end{cases}
\]
and our lemma follows immediately.
\end{proof}

\begin{corollary}
\label{ordcoho}
The following $R$-algebras are isomorphic:
\begin{enumerate}[1.]
\item $(R\llbracket \CDiv_T(X) \rrbracket_F/I_\Sigma) \otimes_{R\llbracket T^* \rrbracket_F} R$.

\item $R\llbracket \CDiv_T(X) \rrbracket_F/J_\Sigma$.

\item $R\llbracket \Pic(X) \rrbracket_F/\overline{I_\Sigma}$, where $\overline{I_\Sigma}$ is the image of $I_\Sigma$ under the surjective homomorphism $R\llbracket \CDiv_T(X) \rrbracket_F \rightarrow R\llbracket \Pic(X) \rrbracket_F$.
\end{enumerate}

The $R$-algebras above are our algebraic model of $\mathsf{h}(X)^\wedge$.
\end{corollary}

\begin{remark}
If $F$ is a polynomial formal group law, the subring $R[T_\sigma^*]_F$ is an algebraic substitute of the equivariant cohomology ring $\mathsf{h}_T(U_\sigma)$ by homotopy invariance and the fact that $X$ is smooth. Then we want to ``glue" $R [T_\tau^*]_F$ over all maximal cones $\tau \in \Sigma_{max}$ together. By the definition of $R[T_\sigma^*]_F$ and Remark \ref{gluing}, the $R[T^*]_F$-algebra obtained from ``gluing" is $R[\CDiv_T(X)]_F/I_\Sigma$, which will be our algebraic model of $\mathsf{h}_T(X)$. As a result, $(R[\CDiv_T(X) ]_F/I_\Sigma) \otimes_{R[T^*]_F} R \cong R[\CDiv_T(X)]_F/J_\Sigma$ is an algebraic model of $\mathsf{h}(X)$.
\end{remark}

\begin{example}
(see \cite[Example 8.3]{KU}) As a first example we let $F$ to be any formal group law, and we consider $X = \P^n$, where $\Sigma \subseteq T_* \cong \Z^n$ is the complete fan consisting of the $n + 1$ rays $\rho_1, \ldots, \rho_{n+1}$ generated by $v_1 = e_1, \ldots, v_n = e_n, v_{n+1} = - e_1 - \cdots - e_n$. Here $\{e_1, \ldots, e_n\}$ is the standard basis of $\Z^n$. Then it is easy to see that 
\[
R\llbracket \CDiv_T(X) \rrbracket_F/I_\Sigma \cong R\llbracket x_1, \ldots, x_{n+1} \rrbracket / \langle x_1 \cdots x_{n+1} \rangle.
\]
Although the right hand side is independent of the formal group law $F$, the isomorphism depends on $F$, see \cite[Remark 2.14]{CPZ}.

Let $\{\alpha_1, \ldots, \alpha_n\}$ be the basis of $T^*$ dual to $\{e_1, \ldots, e_n\}$. The $R\llbracket T^* \rrbracket_F$-algebra structure of $R\llbracket x_1, \ldots, x_{n+1} \rrbracket / \langle x_1 \cdots x_{n+1} \rangle$ is given by
\begin{align*}
x_{\alpha_i} \mapsto \sum_{j = 1}^{n+1} [\langle \alpha_i, v_j \rangle]_F x_j &= x_i -_F x_{n+1} \\
&= x_i - x_{n+1} + x_i x_{n+1} f(x_i,x_{n+1}) \\
&= x_i - u_i x_{n+1}
\end{align*}
where $f$ is some power series determined by the formal group law $F$, and $u_i = 1 - x_i f(x_i, x_{n+1})$ is a unit in $R\llbracket x_1, \ldots, x_{n+1} \rrbracket / \langle x_1 \cdots x_{n+1} \rangle$. Therefore
\begin{align*}
R\llbracket \CDiv_T(X) \rrbracket_F/J_\Sigma &\cong \frac{R\llbracket x_1, \ldots, x_{n+1} \rrbracket}{\langle x_1 \cdots x_{n+1}, x_1 - u_1 x_{n+1}, \ldots, x_n - u_n x_{n+1} \rangle} \\
&\cong R\llbracket x_{n+1} \rrbracket / \langle x_{n+1}^{n+1} \rangle \\
&\cong R[ x_{n+1} ] / \langle x_{n+1}^{n+1} \rangle .
\end{align*}
\end{example}

\section{Comparison Results}

In the present section we compute the formal group rings of a smooth toric variety $X$ for different formal group laws. Then we compare them with known results of equivariant cohomology of smooth toric varieties to justify the validity of our models, and we also obtain new results.

\begin{example}
\label{exchow}
When $F$ is the additive formal group law $F(x,y) = x + y$ over $R$, we recall that $R[M]_F$ is isomorphic to the ring of symmetric powers $\Sym_R(M)$ over $R$. The corresponding oriented cohomology theory is the Chow ring of algebraic cycles modulo rational equivalence, with coefficient ring $CH^*(\pt) = \Z$.

Take $R = \Z$. For every maximal cone $\tau$, $\Z[T_\tau^*]_F$ is isomorphic to $\Sym_{\Z}(T_\tau^*)$, which can be viewed as the ring of integral polynomial functions on $\tau$. Then the above ``gluing" process means that the following $\Sym_{\Z}(T^*)$-algebras are isomorphic:
\begin{enumerate}[(a)]
\item $\Z[\CDiv_T(X)]_F/I_\Sigma$.

\item $\Z[x_{D_\rho}]/\langle \prod_{\rho \in S} x_{D_\rho} \rangle$, where $x_{D_\rho}$ are indeterminates over $\rho \in \Sigma(1)$, and the ideal is generated over all subsets $S \subseteq \Sigma(1)$ such that $S \nsubseteq \sigma(1)$ for any cone $\sigma$.

\item the algebra of integral piecewise polynomial functions on $\Sigma$.
\end{enumerate}
The isomorphism between (b) and (c) is given by mapping $x_{D_\rho}$ to the unique piecewise polynomial function $\varphi_\rho$ satisfying
\begin{enumerate}[(i)]
\item $\varphi_\rho$ is homogeneous of degree 1, 

\item $\varphi_\rho(v_\rho) = 1$, $\varphi_\rho(v_{\rho'}) = 0$ for all $\rho' \in \Sigma(1), \rho' \neq \rho$.
\end{enumerate}
This coincides with the description of the equivariant Chow ring $CH^*_T(X)$ by \cite{B2} and \cite{P}.
\end{example}

\begin{remark}
It is known that for any smooth variety $X$, the natural homomorphism $\Pic(X) \rightarrow CH_{n-1}(X)$ is an isomorphism, where $CH_{n-1}(X)$ is the (usual) Chow group of $(n-1)$-cycles modulo rational equivalence. When $X$ is a smooth toric variety, this isomorphism can be recovered as follows:

First note that since $X$ is smooth, we have $CH^1(X) = CH_{n-1}(X)$. When $F$ is the additive formal group law, we can modify the proof of Lemma \ref{lemmapicgp} to show that $R[\CDiv_T(X)]_F \otimes_{R[T^*]_F} R \cong R[\Pic(X)]_F$. Therefore $\Z[\Pic(X)]_F/\overline{I_\Sigma}$ coincides with $CH^*_T(X) \otimes_{\Sym_{\Z}(T^*)} \Z = CH^*(X)$, the usual Chow ring of $X$. Then our result follows by comparing the degree 1 elements of the two graded rings, where $R[\Pic(X)]_F/\overline{I_\Sigma}$ is given the natural grading as a quotient of a polynomial ring.
\end{remark}

\begin{example}
When $F$ is the multiplicative periodic formal group law $F(x,y) = x + y - \beta xy$ over $R$, where $\beta \in R^\times$, we have seen that $R[M]_F$ is isomorphic to the group ring $R[M] = \{\sum r_i e^{\lambda_i} \mbox{ } | \mbox { } r_i \in R, \lambda_i \in M\}$. The corresponding oriented cohomology theory is the $K$-theory that assigns every smooth variety $Y$ to $K^0(Y)[\beta, \beta^{-1}]$, where $K^0(Y)$ denotes the Grothendieck group of vector bundles on $Y$. The coefficient ring is $K^0(\pt)[\beta, \beta^{-1}] = \Z[\beta, \beta^{-1}]$.

Take $R = \Z$ and $\beta = 1$. For every maximal cone $\tau$, $\Z[T_\tau^*]_F$ can be viewed as the ring of integral exponential functions on $\tau$. Therefore the following $\Z[T^*]$-algebras are isomorphic:
\begin{enumerate}[(a)]
\item $\Z[\CDiv_T(X)]_F/I_\Sigma$.

\item $\Z[e^{\pm D_\rho}]/\langle \prod_{\rho \in S} (1 - e^{D_\rho}) \rangle$, where the ideal is generated over all subsets $S \subseteq \Sigma(1)$ such that $S \nsubseteq \sigma(1)$ for any cone $\sigma$.

\item the algebra of integral piecewise exponential functions on $\Sigma$.
\end{enumerate}
The isomorphism between (b) and (c) is given by mapping $1 - e^{D_\rho}$ to the piecewise function $1 - e^{\varphi_\rho}$, where the notation means that on each cone $\sigma \in \Sigma$,
\[
(1 - e^{\varphi_\rho})_\sigma = 1 - e^{(\varphi_\rho)_\sigma},
\]
and $\varphi_\rho$ is the piecewise polynomial function defined in the previous example. This description agrees with that of the Grothendieck group of equivariant vector bundles $K^0_T(X)$ by \cite{AP} and \cite[Theorem 6.4]{VV}.
\end{example}

\begin{example}
Let $F$ be the multiplicative formal group law over $R$, given by $F(x,y) = x + y - v xy$, where $v$ is not required to be a unit. If $v = \beta \in R^\times$, then clearly we obtain the multiplicative periodic formal group law of the previous example. If $v \notin R^\times$, then the multiplicative formal group law is non-periodic. In particular, if $v = 0$ we get the additive formal group law.

The oriented cohomology theory corresponding to $F$ is the connective $K$-theory. It is the universal oriented cohomology theory for Chow ring and $K$-theory, by specializing at $v = 0$ and $v = \beta \in R^\times$ respectively. The coefficient ring for the connective $K$-theory is $\Z[v]$.

The following construction is motivated by the result in \cite{G}. Consider the group ring $R[M] = \{\sum r_i e^{\lambda_i} \mbox{ } | \mbox { } r_i \in R, \lambda_i \in M\}$, and let $\tr: R[M] \rightarrow R$ be the trace map, the $R$-linear map defined by mapping any $e^\lambda$ to 1. The ideal $\mathfrak{I} = \ker(\tr)$ is generated by $1 - e^\lambda$ over $\lambda \in M$. Then we consider the Rees ring of $R[M]$ with respect to $\mathfrak{I}$
\[
\mathfrak{R} = \Rees(R[M],\mathfrak{I}) = \sum_{n = -\infty}^{\infty} \mathfrak{I}^n t^{-n} = R[M][t, \mathfrak{I} t^{-1}] \subseteq R[M][t, t^{-1}]
\]
where $t$ is an indeterminate, and $\mathfrak{I}^n = R[M]$ if $n \leq 0$. We have the $R$-algebra isomorphisms
\[
R[M]_F \cong \mathfrak{R}/(t - v)\mathfrak{R} \mbox{ and } R\llbracket M \rrbracket_F \cong (\mathfrak{R}/(t - v)\mathfrak{R})^\wedge
\]
induced by $x_\lambda \mapsto \overline{(1 - e^\lambda)t^{-1}}$, and $\overline{e^\lambda} \mapsto 1 - vx_\lambda, \overline{(1 - e^\lambda)t^{-1}} \mapsto x_\lambda$. Here the bar means the image of an element in the quotient ring $\mathfrak{R}/(t - v)\mathfrak{R}$, and $(\mathfrak{R}/(t - v)\mathfrak{R})^\wedge$ is the completion of $\mathfrak{R}/(t - v)\mathfrak{R}$ at the ideal generated by $\overline{(1 - e^\lambda)t^{-1}}$. Specializing at $v = 0$ and $v = \beta \in R^\times$, we have
\begin{align*}
\mathfrak{R}/t \mathfrak{R} &\cong \gr_{\mathfrak{I}} R[M] \\
\mathfrak{R}/(t - \beta) \mathfrak{R} &\cong R[M]
\end{align*}
where $\gr_{\mathfrak{I}} R[M]$ is the associated graded ring of $R[M]$ with respect to $\mathfrak{I}$. Notice that $\gr_{\mathfrak{I}} R[M]$ is also isomorphic to $\Sym_R(M)$ via $1 - e^\lambda \mapsto \lambda$. Therefore we recover the previous two examples.

As a simple, concrete example for the case $v \notin R^\times$ and $v \neq 0$, consider $R = \Z$, $v = 2$, and $M = \Z$. By direct computation we see that 
\[
\Z[\Z]_F \mbox{ } \cong \mbox{ } \Rees(\Z[\Z],\mathfrak{I})/(t - 2)\Rees(\Z[\Z],\mathfrak{I}) \mbox{ } \cong \mbox{ } \Z[x,x']/\langle x + x' - 2xx' \rangle
\]
where the second isomorphism is induced by $\overline{(1 - e^1)t^{-1}} \mapsto x, \overline{(1 - e^{-1})t^{-1}} \mapsto x'$.

Back to our study of toric varieties. Our ``gluing" process above shows that $\Z[v][\CDiv_T(X)]_F/I_\Sigma$ is isomorphic to a ring of tuples of elements in the quotient of Rees rings, where the compatability condition for the tuples of elements hold. This provides a conjecture for the equivariant connective $K$-theory ring of $X$.
\end{example}

\begin{example}
Let $\charact(R) \neq 2$, and let $F$ be the Lorentz formal group law over $R$, given by 
\[
F(x,y) = \frac{x + y}{1 + u^2 xy}
\]
for some $u \in R$. If $u = 0$, then we just recover the additive formal group law. If $u \neq 0$, then $F$ is an elliptic formal group law, and the corresponding oriented cohomology theory is an elliptic cohomology with coefficient ring $\Z[u^2]$. We remark that $F$ also appears in the theory of special relativity as the formula of relativistic addition of parallel velocities, where $u$ is taken to be $\frac{1}{c}$, the reciprocal of the speed of light. 

Even though $F$ is not a polynomial formal group law, we can still study $R[M]_F$. Once again we consider the group ring $R[M]$. Denote by $S$ the multiplicative subset of $R[M]$ generated by $e^\lambda + e^{- \lambda}$ over all $\lambda \in M$, and we let $\mathfrak{H}$ to be the $R$-subalgebra of $S^{-1}R[M]$ generated by $1$ and the ``hyperbolic tangents'' $\frac{e^\lambda - e^{- \lambda}}{e^\lambda + e^{- \lambda}}$ over all $\lambda \in M$. Finally we let $\mathfrak{I}$ be the ideal of $\mathfrak{H}$ generated by all $\frac{e^\lambda - e^{- \lambda}}{e^\lambda + e^{- \lambda}}$, and similar to the previous example we consider the Rees ring of $\mathfrak{H}$ with respect to $\mathfrak{I}$
\[
\mathfrak{R} = \Rees(\mathfrak{H},\mathfrak{I}) = \sum_{n = -\infty}^{\infty} \mathfrak{I}^n t^{-n} = \mathfrak{H}[t, t^{-1}\mathfrak{I}] \subseteq \mathfrak{H}[t, t^{-1}].
\]
We have the $R$-algebra isomorphisms
\[
R[M]_F \cong \mathfrak{R}/(t - u)\mathfrak{R} \mbox{ and } R\llbracket M \rrbracket_F \cong (\mathfrak{R}/(t - u)\mathfrak{R})^\wedge
\]
induced by 
\[
x_\lambda \longmapsto \overline{\frac{e^\lambda - e^{- \lambda}}{e^\lambda + e^{- \lambda}}t^{-1}}
\]
and 
\[
\overline{\frac{e^\lambda - e^{- \lambda}}{e^\lambda + e^{- \lambda}}} \longmapsto u x_\lambda, \mbox{ } \overline{\frac{e^\lambda - e^{- \lambda}}{e^\lambda + e^{- \lambda}}t^{-1}} \longmapsto x_\lambda,
\]
where $\mathfrak{R}/(t - u)\mathfrak{R})^\wedge$ is the completion at the ideal generated by $\overline{\frac{e^\lambda - e^{- \lambda}}{e^\lambda + e^{- \lambda}}t^{-1}}$. Specializing at $u = 0$ and $u = \beta \in R^\times$, we get
\begin{align*}
\mathfrak{R}/t \mathfrak{R} &\cong \gr_{\mathfrak{I}} \mathfrak{H} \\
\mathfrak{R}/(t - \beta) \mathfrak{R} &\cong \mathfrak{H}.
\end{align*}
Notice that the associated graded ring $\gr_{\mathfrak{I}} \mathfrak{H}$ is naturally isomorphic to the ring of symmetric powers $\Sym_R(M)$ by construction, hence we recover the example for the additive formal group law again.

Just like the previous example, we see that for a smooth toric variety $X$ and the Lorentz formal group law $F$, $\Z[u^2][\CDiv_T(X)]_F/I_\Sigma$ is isomorphic to a ring of tuples of elements in the quotient of Rees rings, where the compatability condition for the tuples of elements hold.
\end{example}

\begin{example}
\label{ufgl}
When $F$ is the universal formal group law, the corresponding oriented cohomology theory is the algebraic cobordism (defined over a base field of characteristic 0). The coefficient ring is the Lazard ring $\L$. 

Similar to the first two examples, our ``gluing" process shows that the following $\L \llbracket T^* \rrbracket_F$-algebras are isomorphic:
\begin{enumerate}[(a)]
\item $\L \llbracket \CDiv_T(X) \rrbracket_F/I_\Sigma$.

\item $\L \llbracket x_{D_\rho} \rrbracket/\langle \prod_{\rho \in S} x_{D_\rho} \rangle$, where $x_{D_\rho}$ are indeterminates over $\rho \in \Sigma(1)$, and the ideal is generated over all subsets $S \subseteq \Sigma(1)$ such that $S \nsubseteq \sigma(1)$ for any cone $\sigma$.

\item the algebra of piecewise power series on $\Sigma$ with coefficients in $\L$.
\end{enumerate}
Our result agrees with the description of the equivariant cobordism ring $\Omega^*_T(X)$ in \cite{GK} and \cite{KU}.
\end{example}

To conclude this section we have a concrete example demonstrating how the above comparison results apply.
 
\begin{example}
Consider the del Pezzo surface of degree 6 $dP_6$, obtained by blowing-up the three $T$-fixed points $p_1, p_2, p_3$ of $\P^2$. We begin by recalling some classical results of $dP_6$ (for example, see \cite{M}). The fan $\Sigma$ of $dP_6$, $\Sigma \subseteq T_* \cong \Z^2$, consists of six 2-dimensional maximal cones and all their faces, and the rays are generated by $(0,1), (1,1), (1,0), (0,-1), (-1,-1), (-1,0)$:
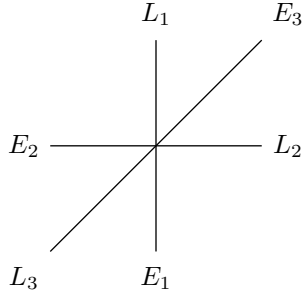
\begin{center}
\begin{pspicture}(-55,-55)(55,55)
\psline(0,0)(40,0)
\psline(0,0)(40,40)
\psline(0,0)(0,40)
\psline(0,0)(0,-40)
\psline(0,0)(-40,-40)
\psline(0,0)(-40,0)
\rput(0,50){$L_1$}
\rput(50,50){$E_3$}
\rput(50,0){$L_2$}
\rput(0,-50){$E_1$}
\rput(-50,-50){$L_3$}
\rput(-50,0){$E_2$}
\end{pspicture}
\end{center}
The $T$-invariant divisors corresponding to the six rays are precisely the six exceptional curves $E_1, E_2, E_3, L_1, L_2, L_3$ on $dP_6$: $E_i$ is the exceptional curve induced by blowing-up $p_i$, and $L_i$ is the strict transform of the unique line in $\P^2$ passing through $p_j$ and $p_k$, $i, j, k$ all distinct.

$\{E_1, E_2, E_3, L_1, L_2, L_3\}$ is a basis of the lattice $\CDiv_T(dP_6)$, and $T^*$ injects into $\CDiv_T(dP_6)$ by
\begin{align*}
x &\longmapsto E_3 + L_2 - L_3 - E_2 \\
y &\longmapsto L_1 + E_3 - E_1 - L_3 
\end{align*}
where $\{x, y\}$ is the basis of $T^*$ dual to the standard basis $\{(1,0), (0,1)\}$ of $T_* \cong \Z^2$. It follows that $\Pic(dP_6)$ is a rank 4 lattice with basis $\{\ell, E_1, E_2, E_3\}$, where 
\[
\ell = L_1 + E_2 + E_3 = L_2 + E_3 + E_1 = L_3 + E_1 + E_2.
\]
The intersection pairing $\langle \mbox{ }, \mbox{ } \rangle: \Pic(dP_6) \times \Pic(dP_6) \rightarrow \Z$ is determined by 
\[
\langle \ell, \ell \rangle = 1, \mbox{ } \langle \ell, E_i \rangle = 0, \mbox{ } \langle E_i, E_j \rangle = -\delta_{ij}
\]
where $\delta_{ij}$ is the usual Kronecker delta function.

Now we go back to our study of equivariant cohomology on $dP_6$. For all the expressions below, we always assume the subindices $i, j \in \{1, 2, 3\}, i \neq j$.

\vspace{3mm}
\noindent
\textbf{(i)} When $F$ is the additive formal group law over $\Z$, the following $\Z[x,y]$-algebras are isomorphic:
\begin{enumerate}[(a)]
\item $CH^*_T(dP_6)$.

\item $\Z[\CDiv_T(dP_6)]_F/I_\Sigma$.

\item $\Z[L_1, L_2, L_3, E_1, E_2, E_3]/\langle L_i L_j, E_i E_j, L_i E_i \rangle$.

\item the algebra of integral piecewise polynomial functions on $\Sigma$.
\end{enumerate}

As a corollary, the following rings are isomorphic:
\begin{enumerate}[(a)]
\item $CH^*(dP_6)$.

\item $(\Z[\CDiv_T(dP_6)]_F/I_\Sigma) \otimes_{\Z[x,y]} \Z$.

\item $\Z[\Pic(dP_6)]_F/\overline{I_\Sigma}$.

\item $\Z[\ell, E_1, E_2, E_3]/\langle \ell^2 + E_i^2, E_i E_j, \ell E_i \rangle$. 
\end{enumerate}
Notice that the relations $\ell^2 + E_i^2 = E_i E_j = \ell E_i = 0$ in the last ring agree with the values of the intersection pairing on $\Pic(dP_6)$.

\vspace{3mm}
\noindent
\textbf{(ii)} When $F$ is the multiplicative periodic formal group law over $\Z$ with $\beta = 1$, the following $\Z[e^{\pm x}, e^{\pm y}]$-algebras are isomorphic:
\begin{enumerate}[(a)]
\item $K^0_T(dP_6)$.

\item $\Z[\CDiv_T(dP_6)]_F/I_\Sigma$.

\item $\displaystyle \frac{\Z[e^{\pm L_1}, e^{\pm L_2}, e^{\pm L_3}, e^{\pm E_1}, e^{\pm E_2}, e^{\pm E_3}]}{\langle (1 - e^{L_i})(1 - e^{L_j}), (1 - e^{E_i})(1 - e^{E_j}), (1 - e^{L_i})(1 - e^{E_i}) \rangle}$.

\item the algebra of integral piecewise exponential functions on $\Sigma$.
\end{enumerate}

Similar to part (i), we see that the following rings are isomorphic:
\begin{enumerate}[(a)]
\item $K^0(dP_6)$.

\item $(\Z[\CDiv_T(dP_6)]_F/I_\Sigma) \otimes_{\Z[e^{\pm x}, e^{\pm y}]} \Z$.

\item $\displaystyle \frac{\Z[e^{\pm \ell}, e^{\pm E_1}, e^{\pm E_2}, e^{\pm E_3}]}{\langle (1 - e^{\ell})^2 + (1 - e^{- E_i})^2, (1 - e^{E_i})(1 - e^{E_j}), (1 - e^{\ell})(1 - e^{E_i}) \rangle}$.
\end{enumerate}

\vspace{3mm}
\noindent
\textbf{(iii)} In general, for arbitrary formal group law $F$ over a ring $R$, we have
\begin{equation}
\label{fgrdp6}
R\llbracket \CDiv_T(dP_6) \rrbracket_F/I_\Sigma \cong \frac{R\llbracket x_{L_1}, x_{L_2}, x_{L_3}, x_{E_1}, x_{E_2}, x_{E_3} \rrbracket}{\langle x_{L_i} x_{L_j}, x_{E_i} x_{E_j}, x_{L_i} x_{E_i} \rangle}.
\end{equation}
There are a couple of useful arithmetic identities in this algebra. Recall that any formal group law $F$ can be expressed as
\[
F(x, y) = x + y - xy \cdot g(x,y)
\]
for some power series $g(x, y)$. Then for example we see that
\[
x_{E_i + E_j} = x_{E_i} +_F x_{E_j} = x_{E_i} + x_{E_j} - x_{E_i} x_{E_j} g(x_{E_i}, x_{E_j}) = x_{E_i} + x_{E_j},
\]
and similarly
\begin{align*}
x_{E_i - E_j} = x_{E_i} -_F x_{E_j} &= x_{E_i} +_F \chi(x_{E_j}) \\
&= x_{E_i} + \chi(x_{E_j}) - x_{E_i} \chi(x_{E_j}) g(x_{E_i}, \chi(x_{E_j})) \\
&= x_{E_i} + x_{-E_j},
\end{align*}
where $\chi(z)$ is the unique power series such that $z +_F \chi(z) = 0$. It follows that we have
\[
(x_{E_i + E_j + E_k})^n = (x_{E_i} + x_{E_j} + x_{E_k})^n = x_{E_i}^n + x_{E_j}^n + x_{E_k}^n  
\]
for any positive integer $n$. Clearly we have the corresponding identities for the $x_{L_i}$'s as well.

Tensoring the isomorphism \eqref{fgrdp6} with $R$ over $R\llbracket T^* \rrbracket_F$, we obtain
\begin{equation}
\label{fgrpicdp6}
R\llbracket \Pic(dP_6) \rrbracket_F/\overline{I_\Sigma} \cong R\llbracket x_{\ell}, x_{E_1}, x_{E_2}, x_{E_3} \rrbracket/\langle x_{\ell}^2 + \chi(x_{E_i})^2, x_{E_i} x_{E_j}, x_{\ell} x_{E_i} \rangle.
\end{equation}
By direct computation we see that 
\[
x_{\ell}^3 = x_{E_1}^3 = x_{E_2}^3 = x_{E_3}^3 = 0
\]
in this ring, which is an expected result as $dP_6$ is a surface. This allows us to simplify the expression in \eqref{fgrpicdp6}. Let $a_{i,j} \in R$ be the coefficients of the formal group law $F$
\[
F(x, y) = \sum_{i,j} a_{i,j} x^i y^j.
\]
Then by \cite[Equation (2.3)]{LM} the power series $\chi(z)$ is of the form
\[
\chi(z) = -z + a_{1,1}z^2 - (a_{1,1})^2 z^3 + \mbox{ terms of degree } \geq 4.
\]
Hence we have
\[
R\llbracket \Pic(dP_6) \rrbracket_F/\overline{I_\Sigma} \cong R\llbracket x_{\ell}, x_{E_1}, x_{E_2}, x_{E_3} \rrbracket/\langle x_{\ell}^2 + x_{E_i}^2, x_{E_i} x_{E_j}, x_{\ell} x_{E_i} \rangle.
\] 
We remark that the isomorphism depends on the formal group law $F$, even though the right hand side is independent of $F$.
\end{example}

\section{Pull-back and push-forward formula of blow-up}

Let $X$ be the smooth toric variety of the fan $\Sigma$, $\sigma \in \Sigma \backslash \Sigma(1)$ be a cone, and $X'$ be the blow-up of $X$ along the orbit closure $\overline{O_\sigma}$. Then $X'$ is smooth toric variety whose fan $\Sigma'$ is equal to the star subdivision of $\Sigma$ relative to $\sigma$,
\[
\Sigma' = \{\theta \in \Sigma \mbox{ }|\mbox{ } \sigma \nsubseteq \theta\} \cup \bigcup_{\sigma \subseteq \theta} \Sigma^*(\theta)
\]
where we let $v_\sigma = \sum_{\rho \in \sigma(1)} v_\rho$, and
\[
\Sigma^*(\theta) = \{\cone(S) \mbox{ }|\mbox{ } S \subseteq \{v_\sigma\} \cup \theta(1), \mbox{ } \sigma(1) \nsubseteq S\}.
\]
The fan $\Sigma'$ is a refinement of $\Sigma$, and the induced toric morphism $\pi: X' \rightarrow X$ is projective. Therefore for every equivariant cohomology theory $\mathsf{h}_T$, $\pi$ induces the pull-back homomorphism $\pi^*: \mathsf{h}_T(X) \rightarrow \mathsf{h}_T(X')$ and also the push-forward homomorphism $\pi_*: \mathsf{h}_T(X') \rightarrow \mathsf{h}_T(X)$. Notice that $\pi^*$ is a $\mathsf{h}_T(\pt)$-algebra homomorphism, and $\pi_*$ is a $\mathsf{h}_T(X)$-module homomorphism, where $\mathsf{h}_T(X)$ acts on $\mathsf{h}_T(X')$ via $\pi^*$. Our goal in this section is to define two homomorphisms of formal group rings that will serve as algebraic substitutes of the pull-back and the push-forward homomorphisms. We remark that formulas for the pull-back and the push-forward of equivariant Chow rings are proved in \cite[Theorem 2.3]{B1}.

First, let $E$ be the $T$-invariant prime divisor on $X'$ corresponding to the ray $\widetilde{\rho}$ generated by $v_\sigma$. Then we have
\[
\CDiv_T(X') \cong \CDiv_T(X) \oplus \Z \cdot E,
\]
where the $T$-invariant prime divisor $D_{\rho, \Sigma'}$ on $X'$ corresponding to $\rho \in \Sigma'$ is the strict transform of the divisor $D_{\rho, \Sigma}$ on $X$ corresponding to $\rho \in \Sigma$. We denote both $D_{\rho, \Sigma'}$ and $D_{\rho, \Sigma}$ by $D_\rho$ if the ambient fan is clear from the context. 

Next, we want to define the pull-back homomorphism for formal group rings. Informally speaking from the point of view of piecewise functions on fans, $\pi^*: \mathsf{h}_T(X) \rightarrow \mathsf{h}_T(X')$ is given by treating piecewise functions on $\Sigma$ as piecewise functions on the refinement $\Sigma'$. Translating this to the language of formal group rings, we define
\[
\begin{tabular}{rcl}
$\pi^*: R\llbracket \CDiv_T(X) \rrbracket_F/I_\Sigma$ & $\longrightarrow$ & $R\llbracket \CDiv_T(X') \rrbracket_F/I_{\Sigma'}$ \\
$x_{D_\rho}$ & $\longmapsto$ & $\begin{cases} x_{D_\rho} & \mbox{if } \rho \notin \sigma(1) \\ x_{D_\rho + E} = x_{D_\rho} +_F x_E & \mbox{if } \rho \in \sigma(1). \end{cases}$ \\
\end{tabular}
\]
The underlying geometric meaning can again be explained by the Orbit-Cone Correspondence: $\rho \in \sigma(1)$ if and only if $\overline{O_\sigma} \subseteq D_\rho$.

For the push-forward homomorphism, we impose the condition that $F$ is the associated formal group law of a birationally invariant theory $\mathsf{h}$, i.e the push-forward of the fundamental class satisfies $f_*(1_Y) = 1_X$ for any birational projective morphism $f: Y \rightarrow X$ between smooth irreducible varieties. Examples of birationally invariant theories include Chow ring over an arbitrary field, $K$-theory over a field of characteristic 0. It is known that the connective $K$-theory over a field of characteristic 0 is univeral among all birationally invariant theories, see \cite[Theorem 4.3.9]{LM} and \cite[Example 8.10]{CPZ}. Therefore from now on we assume $F$ is of the form $F(x,y) = x + y - vxy$ for some $v \in R$.

Now we begin the construction of the push-forward $\pi_*: R\llbracket \CDiv_T(X') \rrbracket_F/I_{\Sigma'} \rightarrow R\llbracket \CDiv_T(X) \rrbracket_F/I_\Sigma$, which is a homomorphism of $R\llbracket \CDiv_T(X) \rrbracket_F/I_\Sigma$-modules. As the blow-up morphism $\pi$ is birational and projective, by our assumption on $F$ we define $\pi_*(1) = 1$. Since $D_{\rho, \Sigma'}$ is the strict transform of $D_{\rho, \Sigma}$, we define $\pi_*(x_{D_{\rho, \Sigma'}}) = x_{D_{\rho, \Sigma}}$.

From here we can deduce that $\pi_*(-_F x_{E}) = 0$, and the proof is as follows: Let $\rho \in \sigma(1)$ be a ray in $\sigma \subseteq \Sigma$, and consider $x_{D_{\rho, \Sigma'}} +_F x_E = \pi^*(x_{D_{\rho, \Sigma}})$. By the projection formula, 
\begin{equation}
\label{projform}
\pi_*(x_{D_{\rho, \Sigma'}} +_F x_E) = \pi_*(\pi^*(x_{D_{\rho, \Sigma}})) = \pi_*(1) x_{D_{\rho, \Sigma}} = x_{D_{\rho, \Sigma}}.
\end{equation}
On the other hand, 
\[
\pi_*\Big((x_{D_{\rho, \Sigma'}} +_F x_E) -_F x_E\Big) = \pi_*(x_{D_{\rho, \Sigma'}}) = x_{D_{\rho, \Sigma}}
\]
as well. By subtracting the two equations, we get 
\begin{align*}
\pi_*\Big((-_F x_E) - v (x_{D_{\rho, \Sigma'}} +_F x_E)(-_F x_E)\Big) &= \pi_*\Big(\big(1 - v (x_{D_{\rho, \Sigma'}} +_F x_E)\big)(-_F x_E)\Big) \\
&= \pi_*\Big(\pi^*\big(1 - v x_{D_{\rho, \Sigma}}\big)(-_F x_E)\Big) \\
&= (1 - v x_{D_{\rho, \Sigma}}) \pi_*(-_F x_E) \\
&= 0.
\end{align*}
As $1 - v x_{D_{\rho, \Sigma}}$ is a unit in $R\llbracket \CDiv_T(X) \rrbracket_F/I_\Sigma$, we have $\pi_*(-_F x_{E}) = 0$. Notice that in general $\pi_*(x_{E}) \neq 0$.

If $\dim \sigma = 2$, then by the property of a $R\llbracket \CDiv_T(X) \rrbracket_F/I_\Sigma$-module homomorphism we can already determine the image of every element in $R\llbracket \CDiv_T(X') \rrbracket_F/I_{\Sigma'}$. Let $\sigma(1) = \{\rho_1, \rho_2\} \subseteq \Sigma(1)$. After the star subdivision $\{\rho_1, \rho_2\} \nsubseteq \theta(1)$ for any cone $\theta \in \Sigma'$, hence $x_{D_{\rho_1, \Sigma'}} x_{D_{\rho_2, \Sigma'}} = 0$ in $R\llbracket \CDiv_T(X') \rrbracket_F/I_{\Sigma'}$. Then for example,
\begin{align}
\label{dim2eq} \pi^*(x_{D_{\rho_1, \Sigma}}) x_{D_{\rho_2, \Sigma'}} &= (x_{D_{\rho_1, \Sigma'}} + x_E - v x_{D_{\rho_1, \Sigma'}} x_E) x_{D_{\rho_2, \Sigma'}} \\
&= x_E x_{D_{\rho_2, \Sigma'}} \nonumber \\
&\mapsto \pi_*(\pi^*(x_{D_{\rho_1, \Sigma}}) x_{D_{\rho_2, \Sigma'}}) = x_{D_{\rho_1, \Sigma}} x_{D_{\rho_2, \Sigma}}. \nonumber 
\end{align}
It follows from equations \eqref{projform} and \eqref{dim2eq} that
\begin{align*}
\pi_*(x_{D_{\rho_2, \Sigma'}} +_F x_E) &= x_{D_{\rho_2, \Sigma}} + \pi_*(x_E) - v x_{D_{\rho_1, \Sigma}} x_{D_{\rho_2, \Sigma}} \\
&= x_{D_{\rho_2, \Sigma}},
\end{align*}
therefore $\pi_*(x_E) = v x_{D_{\rho_1, \Sigma}} x_{D_{\rho_2, \Sigma}}$.

\begin{theorem}
\label{dim2push}
Let $\dim \sigma = 2$ and $\sigma(1) = \{\rho_1, \rho_2\} \subseteq \Sigma(1)$. We use $x_i$ to denote both $x_{D_{\rho_i, \Sigma'}}$ and $x_{D_{\rho_i, \Sigma}}$, $i = 1, 2$. Then the push-forward homomorphism defined above satisfies
\[
\begin{tabular}{rcl}
$\pi_*: R\llbracket \CDiv_T(X') \rrbracket_F/I_{\Sigma'}$ & $\longrightarrow$ & $R\llbracket \CDiv_T(X) \rrbracket_F/I_\Sigma$ \\
$1$ & $\longmapsto$ & $1$ \\
$x_E$ & $\longmapsto$ & $v x_1 x_2$ \\
$x_E^n$ & $\longmapsto$ & $\displaystyle v \sum_{i=1}^{n} x_1^{n+1-i} x_2^i - \sum_{i=1}^{n-1} x_1^{n-i} x_2^i$ \\
$x_a^s x_E^t$ & $\longmapsto$ & $x_a x_b^t (x_a -_F x_b)^{s-1}$, \\
\end{tabular}
\]
where $n \geq 2, s \geq 1, t \geq 0$, and $\{a, b\} = \{1, 2\}$. 
\end{theorem}

\begin{proof}
First, by induction on $t$ we see that
\begin{align}
\label{lemmaeq} \pi_*(x_a x_E^t) &= \pi_*((x_b + x_E - v x_b x_E)x_a x_E^{t-1}) \\
&= \pi_*(\pi^*(x_b) x_a x_E^{t-1}) \nonumber \\
&= x_a x_b^t, \nonumber 
\end{align}
where the first equality follows from the fact that $x_a x_b = 0$ in $R\llbracket \CDiv_T(X') \rrbracket_F/I_{\Sigma'}$.

To compute the image of $x_E^n$, we use induction on $n$, and compare the image of $\pi^*(x_a) x_E^{n-1}$ computed by two different methods:
\[
\pi_*(\pi^*(x_a) x_E^{n-1}) = x_a \pi_*(x_E^{n-1})
\]
and
\[
\pi_*(\pi^*(x_a) x_E^{n-1}) = \pi_*((x_a + x_E - v x_a x_E) x_E^{n-1}) = x_a x_b^{n-1} + \pi_*(x_E^n) - v x_a x_b^n.
\]

Next we want to compute the image of $x_a^2$. Let $\chi(z)$ be the unique power series such that $z +_F \chi(z) = 0$. Then
\begin{align*}
\pi_*(x_a^2) &= \pi_*((\pi^*(x_a) -_F x_E) x_a) \\
&= \pi_*((\pi^*(x_a) + \chi(x_E) - v \pi^*(x_a) \chi(x_E)) x_a) \\
&= x_a^2 + \chi(x_b) x_a - v x_a \chi(x_b) x_a \\ 
&= x_a (x_a -_F x_b),
\end{align*} 
where the third equality follows from equation \eqref{lemmaeq}. By using the same trick for $\pi_*(x_a x_E^t)$, we see that 
\[
\pi_*(x_a^2 x_E^t) = x_a x_b^t (x_a -_F x_b)
\]
for every $t \geq 0$. This allows us to compute $\pi_*(x_a^3)$, and then $\pi_*(x_a^3 x_E^t)$, by the same argument as above. The general case now follows from induction.
\end{proof}

\begin{remark}
In our case where the formal group law $F$ is of the form $F(x,y) = x + y - vxy$, we have an explicit description of the power series $\chi(z)$:
\[
\chi(z) = -z \sum_{i = 0}^\infty (vz)^i = - \sum_{i = 0}^\infty v^i z^{i+1}.
\]
Hence $-_F x_E = - \sum_{i = 0}^\infty v^i x_E^{i+1}$. Then by using the formula in Proposition \ref{dim2push}, we recover the result $\pi_*(-_F x_E) = 0$ when $\dim \sigma = 2$.
\end{remark}

If $\dim \sigma = 3$, let $\sigma(1) = \{\rho_1, \rho_2, \rho_3\} \subseteq \Sigma(1)$ and use $x_i$ to denote both $x_{D_{\rho_i, \Sigma'}}$ and $x_{D_{\rho_i, \Sigma}}$, $i = 1, 2, 3$. We further define 
\[
\pi_*(x_i x_j) = x_i x_j
\]
for every $i, j \in \{1, 2, 3\}$ such that $i \neq j$. Then the image for the rest of the elements in $R\llbracket \CDiv_T(X') \rrbracket_F/I_{\Sigma'}$ can be computed by the same technique as above. For example, $\pi_*(x_i x_j x_E) = x_1 x_2 x_3$, $\pi_*(x_i x_E) = v x_1 x_2 x_3$, and $\pi_*(x_E) = v^2 x_1 x_2 x_3$. Finally, we remark that this process of defining the push-forward homomorphism $\pi_*$ extends naturally for arbitrary dimension of $\sigma$.  

\vspace{3mm}

\noindent
\textit{Acknowledgments:} The author thanks Kirill Zainoulline for his suggestion on this project and his helpful comments, and Alexander Duncan for his help on toric varieties.


\begin{thebibliography}{9}
\bibitem{AHW} S. Au, M. Huang and M. E. Walker, \emph{The Equivariant K-Theory of Toric Varieties}, Journal of Pure and Applied Algebra \textbf{213} (2009), Issue 5, 840 - 845.

\bibitem{AP} D. Anderson and S. Payne, \emph{Operational K-theory}, arXiv:1301.0425 [math.AG].

\bibitem{B1} M. Brion, \emph{Piecewise Polynomial Functions, Convex Polytopes and Enumerative Geometry}, Banach Center Publications \textbf{36} (1996), Issue 1, 25 - 44.

\bibitem{B2} M. Brion, \emph{Equivariant Chow Groups for Torus Actions}, Transformation Groups \textbf{2} (1997), no. 3, 225 - 267.

\bibitem{CLS} D. A. Cox, J. B. Little and H. K. Schenck, \emph{Toric Varieties}, volume 124 of \emph{Graduate Studies in Mathematics}, American Mathematical Society, Providence, RI, 2011. 

\bibitem{CPZ} B. Calm\`{e}s, V. Petrov and K. Zainoulline, \emph{Invariants, Torsion Indices and Oriented Cohomology of Complete Flags}, Annales scientifiques de l'\'{E}NS \textbf{46}, fascicule 3 (2013), 405 - 448.

\bibitem{CZZ1} B. Calm\`{e}s, K. Zainoulline and C. Zhong, \emph{A Coproduct Structure on the Formal Affine Demazure Algebra}, arXiv:1209.1676 [math.RA].

\bibitem{CZZ2} B. Calm\`{e}s, K. Zainoulline and C. Zhong, \emph{Push-Pull Operators on the Formal Affine Demazure Algebra and its Dual}, arXiv:1312.0019 [math.AG].

\bibitem{ELFST} E.J. Elizondo, P. Lima-Filho, F. Sottile and Z. Teitler, \emph{Arithmetic toric varieties}, Mathematische Nachrichten \textbf{287} (2014), Issue 2-3, 216 - 241.

\bibitem{F} W. Fulton, \emph{Introduction to Toric Varieties}, volume 131 of \emph{Annals of Mathematics Studies}, Princeton University Press, Princeton, NJ, 1993.

\bibitem{GK} J. K. Gonz\'{a}lez and K. Karu, \emph{Bivariant Algebraic Cobordism}, arXiv:1301.4210 [math.AG].

\bibitem{G} J.P.C. Greenlees, \emph{Multiplicative Equivariant Formal Group Laws}, Journal of Pure and Applied Algebra \textbf{165} (2001), 183 - 200.

\bibitem{KU} A. Krishna and V. Uma, \emph{The Algebraic Cobordism Ring of Toric Varieties}, International Mathematics Research Notices, rns212, 39 pages.

\bibitem{LM} M. Levine and F. Morel, \emph{Algebraic Cobordism}, Springer Monographs in Mathematics, Springer, New York, NJ, 2007.

\bibitem{M} Y. I. Manin, \emph{Cubic Forms: Algebra, Geometry, Arithmetic}, second edition, volume 4 of \emph{North-Holland Mathematical Library}, North-Holland, Amsterdam, 1986.

\bibitem{Me} A. Merkurjev, \emph{Equivariant K-Theory}, Handbook of K-Theory, volume 2, 925 - 954, Springer, Berlin, 2005.

\bibitem{P} S. Payne, \emph{Equivariant Chow Cohomology of Toric Varieties}, Math. Res. Lett. \textbf{13} (2006), no. 1, 29 - 41.

\bibitem{VV} G. Vezzosi and A. Vistoli, \emph{Higher Algebraic K-theory for Actions of Diagonalizable Groups}, Inventiones mathematicae \textbf{153} (2003), Issue 1, 1 - 44.
\end{thebibliography}
\end{document}